\documentclass[12pt,a4paper]{amsart}
\usepackage{amsmath,amssymb,amscd}
\input xy
\xyoption{all}

\newcommand{\ra}{\rightarrow}

\newcommand{\CC}{\mathbb C}
\newcommand{\ZZ}{\mathbb Z}

\newcommand{\PP}{\mathbb P}
\newcommand{\e}{\epsilon}

\newcommand{\cF}{\mathcal{F}}
\newcommand{\cU}{\mathcal{U}}
\newcommand{\cO}{\mathcal{O}}

\newcommand{\Pic}{\mbox{Pic}}
\newcommand{\Ext}{\mbox{Ext}}
\newcommand{\Gcd}{\mbox{gcd}}
\newcommand{\End}{\mbox{End}}
\newcommand{\Hom}{\mbox{Hom}}

\newcommand{\rk}{\mbox{rk}}
\theoremstyle{plain}
\newtheorem{theorem}{Theorem}[section]
\newtheorem{lem}[theorem]{Lemma}
\newtheorem{prop}[theorem]{Proposition}
\newtheorem{cor}[theorem]{Corollary}
\newtheorem{rem}[theorem]{Remark}

\begin{document}
\title[Coherent systems]{Hodge polynomials and birational types of moduli spaces of coherent systems on elliptic curves}

\author{H. Lange}
\author{P. E. Newstead}

\address{H. Lange\\Mathematisches Institut\\
              Universit\"at Erlangen-N\"urnberg\\
              Bismarckstra\ss e $1\frac{ 1}{2}$\\
              D-$91054$ Erlangen\\
              Germany}
              \email{lange@mi.uni-erlangen.de}
\address{P.E. Newstead\\Department of Mathematical Sciences\\
              University of Liverpool\\
              Peach Street, Liverpool L69 7ZL, UK}
\email{newstead@liv.ac.uk}

\thanks{Both authors are members of the research group VBAC (Vector Bundles on Algebraic Curves). The second author 
would like to thank the Mathematisches Institut der Universit\"at 
         Erlangen-N\"urnberg for its hospitality}
\keywords{Vector bundle, coherent system, moduli space, elliptic curve, Hodge polynomial}
\subjclass[2000]{Primary: 14H60; Secondary: 14F05, 32L10}

\begin{abstract}
In this paper we consider moduli spaces of coherent systems on an elliptic curve. We compute their Hodge polynomials and 
determine their birational types in some cases. Moreover we prove that certain moduli spaces of coherent systems 
are isomorphic. This last result uses the Fourier-Mukai transform of coherent systems introduced by Hern\'andez Ruip\'erez
and Tejero Prieto.  

\end{abstract}
\maketitle

\section{Introduction}

\noindent
A {\it coherent system of type $(n,d,k)$} on a smooth projective curve $C$ over an algebraically closed field is by 
definition a pair $(E,V)$ consisting of
a vector bundle $E$ of rank $n$ and degree $d$ over $C$ and a vector subspace $V \subset H^0(E)$ of dimension $k$. 
For any real number $\alpha$, the {\it $\alpha$-slope} of a coherent system $(E,V)$ of type $(n,d,k)$ is defined by
$$
\mu_{\alpha}(E,V) := \frac{d}{n} + \alpha \frac{k}{n}.
$$
A {\it coherent subsystem} of $(E,V)$ is a coherent system $(E',V')$ such that $E'$ is a subbundle of $E$ and 
$V' \subset V \cap H^0(E')$.
A coherent system $(E,V)$ is called 
{\it $\alpha$-stable} ({\it $\alpha$-semistable}) if
$$
\mu_{\alpha}(E',V') < \mu_{\alpha}(E,V) \ \ (\mu_{\alpha}(E',V') \le \mu_{\alpha}(E,V))
$$
for every proper coherent subsystem $(E',V')$ of $(E,V)$.
The $\alpha$-stable coherent systems of type $(n,d,k)$ on $X$ form a quasi-projective moduli space which we 
denote by $G(\alpha;n,d,k)$.

In a previous paper \cite{ln} we studied the case of an elliptic curve $C$, determined precisely when $G(\alpha;n,d,k)$ is non-empty and showed 
that in this case it is always smooth and irreducible of the expected dimension
$$
\beta(d,k) := k(d-k) + 1.
$$
 
In this paper we consider again the case where $C$ is elliptic and assume for convenience that the base field is $\CC$. 
After summarizing some properties of the Hodge polynomial $\epsilon_X(u,v)$ of a quasi-projective variety $X$ in section 2, we investigate first 
the spaces $G_0(n,d,k)$, defined by $G_0(n,d,k): = G(\alpha;n,d,k)$ for small positive $\alpha$.
When $\Gcd(n,d) = 1$ we prove in particular that for fixed $d$ and $k$ the birational type (Corollary \ref{cor2.2})
and the Hodge polynomial (Corollary \ref{cor2.3}) of $G_0(n,d,k)$ are independent of the rank $n$.
We give also an alternative proof of a special case of the main result of \cite{ht} saying that $G_0(n',d,k) \simeq G_0(n,d,k)$ when 
$n' \equiv n \mod d$ (see Corollary \ref{cor2.4}). Moreover we show\\

\noindent
{\bf Proposition \ref{prop2.6}.}
	{\it If ${\emph\Gcd}(n,d) = {\emph\Gcd}(n',d) = 1$ and $n' \not \equiv n \mod d$, then $G_0(n',d,1) \not \simeq G_0(n,d,1)$ and 
$G_{0}(n',d,d-1) \not \simeq G_{0}(n,d,d-1)$.}\\ 

In section 4 we show that if $d$ and $\Gcd(n,d)$ are fixed, the birational type of $G_0(n,d,1)$ is independent of $n$.
This improves \cite[Theorem 5.2]{ht} in the case $k=1$. When $\Gcd(n,d) = 2$, we describe in section 5 a stratification of 
$G_0(n,d,1)$ and use it to calculate its Hodge polynomial. We obtain similar results for the corresponding 
moduli spaces $G_0(n,N,1)$ with fixed determinant $N$.\\

\noindent
{\bf Theorem \ref{thm4.9}.}\\
(a)
$\e_{G_0(n,d,1)}(u,v) =$
$$ 
= \frac{(1+u)(1+v)(1-(uv)^{\frac{d}{2}})}{(1-uv)^2(1+uv)}[(u+v)(uv -(uv)^{\frac{d}{2}}) + (1+uv)(1-(uv)^{\frac{d}{2}+1})];
$$
(b)
$\e_{G_0(n,N,1)}(u,v) =$
$$ 
= \frac{1-(uv)^{\frac{d}{2}}}{(1-uv)^2(1+uv)}[(u+v)(uv -(uv)^{\frac{d}{2}}) + (1+uv)(1-(uv)^{\frac{d}{2}+1})].
$$

\vspace*{0.3cm}
In section 6 we allow the parameter $\alpha$ to vary and compute the Hodge polynomial of $G(\alpha;2+ad,d,1)$. We recall that there
are only finitely many distinct moduli spaces as $\alpha$ varies, usually labelled $G_i := G_i(2+ad,d,1)$.\\

\noindent
{\bf Theorem \ref{thm5.6}.}
{\it For $i = 0, \ldots, L$ we have}\\ 
$ \displaystyle{\e_{G_i}(u,v) = (1+u)(1+v)\frac{1-(uv)^d}{1-uv} \;+}$
$$
+ \frac{(1+u)^2(1+v)^2(1-(uv)^{\frac{d-\gamma}{2} -i)}}{(1-uv)^2(1-(uv)^2)} 
(uv - (uv)^{\gamma + 2i})(1 - (uv)^{\frac{d-\gamma}{2}-i+1}),
$$
{\it where $\gamma$ is $1$ if $d$ is odd and $2$ if $d$ is even.}

\vspace*{0.3cm}
We note that these Hodge polynomials are independent of $a$. When $i=0$ it is in fact known (see \cite{ht}) that the spaces 
$G_0(2+ad,d,1)$ are all isomorphic for fixed $d$. Our theorem provides evidence that this result may extend to arbitrary $i$
and we prove this in section 7. In section 8 we investigate further the birational type of $G(\alpha;n,d,k)$. We show
that, if $\gcd(n,d)=1$ and $k\le d$ or $\Gcd(n,d) = 2$ and $k=1$ or $\gcd(n-k,d)=1$ and $k<\min(d,n)$, the variety is birational to $\PP^{k(d-k)} \times C$ (Propositions \ref{prop8.1}, \ref{prop6.1}, \ref{prop8.3}). Finally, if $k<d$, $\gcd(n,d)=h>1$ and $S^hC$ denotes the $h$-fold symmetric product of $C$, we prove
\\

\noindent{\bf Theorem \ref{th8.4}.}
{\it For all $\alpha$ for which $G(\alpha;n,d,k)\ne\emptyset$,\\
\emph{(1)} $G(\alpha;n,N,k)$ is birational to a variety $Y_N$, where $Y_N$ is fibred over $\PP^{h-1}$ with general fibre unirational.\\
\emph{(2)} $G(\alpha;n,d,k)$ is birational to a variety $Y$, where $Y$ is fibred over  $S^hC$ with general fibre unirational.}
\\

We are grateful to the referee for a careful reading of the paper.

\section{Hodge polynomials}

We recall the basic properties of Hodge polynomials as defined by Deligne in \cite{d}. 
For any quasi-projective variety $X$ over the field of complex numbers, Deligne defined a mixed Hodge structure on the cohomology groups
$H^k_c(X,\CC)$ with compact support with associated Hodge polynomial $\epsilon_X(u,v)$. When $X$ is a smooth projective
variety, we have 
$$
\epsilon_X(u,v) = \sum_{p,q} h^{p,q}(X)u^pv^q,
$$
where $h^{p,q}(X)$ are the usual Hodge numbers. In particular, in this case  
$\epsilon_X(u,u)$ is the usual Poincar\'e polynomial
$P_X(u)$. We need only the following properties of the Hodge polynomials 
(see \cite{d} and \cite[Theorem 2.2 and Lemmas 2.3 and 2.4]{mov}). 

\begin{itemize}
\item If $X$ is a finite disjoint union $X = \sqcup_i X_i$ of locally closed subvarieties $X_i$, then
$$
\epsilon_X = \sum_{i} \epsilon_{X_i}.
$$
\item If $Y \ra X$ is an algebraic fibre bundle with fibre $F$ which is locally trivial in the Zariski topology, then 
$$
\epsilon_Y = \epsilon_F \cdot \epsilon_X.
$$
\item If $Y \ra X$ is a map between quasi-projective varieties which is a locally trivial fibre bundle 
in the complex topology with fibres projective spaces $F = \PP^N$ for some $N > 0$, then 
$$
\e_Y = \e_F \cdot \e_X.
$$ 
\end{itemize}

Moreover we need the Hodge polynomials of the Grassmannians. In fact,
\begin{equation} \label{eqn5.1}
\epsilon_{{\rm Gr}(r,N)}(u,v) = \frac{(1 - (uv)^{N-r+1})(1- (uv)^{N-r+2}) \cdots (1- (uv)^{N})}{(1-uv)(1-(uv)^2) \cdots (1-(uv)^r)}.
\end{equation}

\section{$G_0(n,d,k)$ for coprime $n$ and $d$}

Let $C$ be an elliptic curve defined over $\CC$ and suppose that $n,d,k$ are integers with $n\ge2$, $1\le k\le d$. 
Let $G_0(n,d,k)$ denote the moduli space of coherent systems of type $(n,d,k)$ which are $\alpha$-stable for small positive $\alpha$ (we call such systems $0^+$-{\it stable}). Similarly for any line bundle $N$ of degree $d$ on $C$ let $G_0(n,N,k)$
denote the corresponding moduli space with fixed determinant $N$.

In this section we assume $\Gcd(n,d) =1$. Then there exists a Poincar\'e bundle $\cU$ on $C \times C$. It has the property that 
$\cU|_{C \times \{E\}} \simeq E$ for all stable bundles $E$ of rank $n$ and degree $d$, where we identify the moduli space of 
stable bundles of type $(n,d)$ with the curve $C$ in the usual way (see \cite{at} and \cite{tu}). 
Let $p_i$ denote the $i$-th projection of $C \times C$.

\begin{prop} \label{prop2.1} If $\emph{\Gcd}(n,d) = 1$, then\\
\emph{(1)} $G_0(n,N,k)$ is isomorphic to the Grassmannian ${\rm Gr}(k,d)$;\\
\emph{(2)} $G_0(n,d,k)$ is a ${\rm Gr}(k,d)$-bundle over $C$.
\end{prop}

\begin{proof}
(1) For $\Gcd(n,d) = 1$ any semistable vector bundle on $C$ is stable. 
For a fixed $N \in \Pic^{d}(C)$ there is a unique stable bundle $E$ of rank $n$ 
and determinant $N$ (this follows from \cite[Theorem 7]{at}). Then $(E,V)$ belongs to $G_0(n,N,k)$ for any $k$-dimensional subspace $V$ of $H^0(E)$.

(2) From part (1) we conclude that $G_0(n,d,k)$ can be identified with the Grassmannian bundle of $k$-planes in the fibres of 
the rank-$d$ vector bundle $p_{2*}\cU$ on $C$.
\end{proof}

\begin{cor} \label{cor2.2}
The birational type of $G_0(n,d,k)$ is independent of $n$ provided $\emph{\Gcd}(n,d) = 1$.
\end{cor}

\begin{proof}
Observe that the Gr$(k,d)$-bundle over $C$ of the proposition is Zariski locally trivial.
\end{proof}

This improves the statement in \cite{ht} that for fixed $d$ and $k$ there are at most $d$ different birational types of
varieties $G(\alpha;n,d,k)$, since we have by \cite[Theorem 4.4 (ii)]{ln} that the birational type is independent of $\alpha$
in this case.

Moreover we conclude that the Hodge polynomial of $G_0(n,d,k)$ is given by

\begin{cor} \label{cor2.3}
Suppose $\emph{\Gcd}(n,d) = 1$. Then
$$
\e_{G_0}(u,v) = \frac{(1 - (uv)^{d-k+1})(1-(uv)^{d-k+2}) \cdots (1-(uv)^d)}
{(1-uv)(1-(uv)^2) \cdots (1-(uv)^k)} (1+u)(1+v).
$$
\begin{flushright} $\square$ \end{flushright} 
\end{cor}

\begin{lem} \label{lem2.3}
$p_{2*}\cU$ is a stable vector bundle of rank $d$ on $C$.
\end{lem}

\begin{proof}
According to \cite{ht}, $\cU$ is the kernel of a Fourier-Mukai transform $\Phi_{\cU}$ on $C \times C$.
Hence $p_{2*}\cU = \Phi_{\cU}(\cO_C)$ is stable by \cite[Proposition 2.8 and Remark 2.9]{ht}.
\end{proof}

The following corollary is a special case of the main result of \cite{ht}.
\begin{cor} \label{cor2.4}
If $\emph{\Gcd}(n,d) = 1$ and $n' \equiv n \mod d$, then $G_0(n',d,k)$ is isomorphic to $G_0(n,d,k)$.
\end{cor}

\begin{proof}
Let $\cU'$ be a Poincar\'e bundle  for $(n',d)$ on $C \times C$. Then by \cite[Proposition 7.2]{ln}, 
$$
c_1(p_{2*}\cU) = s[C] \quad \mbox{and} \quad c_1(p_{2*}\cU') = s'[C]
$$
where $sn \equiv s'n' \equiv -1 \mod d$. Since $n' \equiv n \mod d$, it follows that
$s \equiv s' \mod d$. From Lemma \ref{lem2.3} and the classification of stable bundles on an elliptic curve we conclude that 
$$
p_{2*} \cU \simeq p_{2*} \cU' \otimes M
$$
with $M \in \Pic(C)$. Hence $P(p_{2*}\cU) \simeq P(p_{2*} \cU')$ and the same holds 
for the Grassmannian fibrations. Now the assertion follows from the description of $G_0(n,d,k)$ in the proof of Proposition \ref{prop2.1}.
\end{proof}

Now suppose $\Gcd(n',d) =1$ and $n' \not \equiv n \mod d$. Then in the above argument 
$s' \not \equiv s \mod d$. So the projective bundles $P(p_{2*} \cU)$ and $P(p_{2*} \cU')$ are not 
isomorphic as projective bundles and therefore also not isomorphic as varieties by the argument in \cite[Proposition 8.2]{ln}.
This implies

\begin{prop} \label{prop2.6}
If $\emph{\Gcd}(n,d) = \emph{\Gcd}(n',d) = 1$ and $n' \not \equiv n \mod d$, then $G_0(n',d,1) \not \simeq G_0(n,d,1)$ and 
$G_{0}(n',d,d-1) \not \simeq G_{0}(n,d,d-1)$.  $\square$  
\end{prop}

\section{$G_0(n,d,1)$ for arbitrary $(n,d)$}

In this section we determine the birationality type of the moduli space $G_0(n,d,1)$ for arbitrary $n$ and $d$.
Suppose $h :=\Gcd(n,d)$. Then any semistable vector bundle $E$ of rank $n$ and degree $d$ on $C$ is of the form
$E = E_1 \oplus \cdots \oplus E_{\ell}$ with $E_i$ indecomposable of slope $\frac{d}{n}$ and $\ell \leq h$.
Each $E_i$ is semistable and is the unique indecomposable multiple extension of $\frac{rk\; E_i}{n/h}$ copies of 
a stable bundle $F$ of rank $\frac{n}{h}$
and degree $\frac{d}{h}$. Moreover a generic such $E$ has the form 
\begin{equation} \label{eqn1}
E = F_1 \oplus \cdots \oplus F_h
\end{equation}
with all $F_i$ stable and non-isomorphic of rank $\frac{n}{h}$ and degree $\frac{d}{h}$ (see \cite{tu}).

A generic $(E,V) \in G_0(n,d,1)$ has $E= F_1 \oplus \cdots \oplus F_h$ as a bundle and $V$ is a one-dimensional subspace of 
$H^0(E)$ whose projection to each $F_i$ is still one-dimensional. Denote by $G_0^1$ the open set of coherent 
systems in $G_0(n,d,1)$ which are generic in this sense.

So we get a $(\times_{i=1}^h \PP^{\frac{d}{h}-1})$-bundle over $\times_{i=1}^h C \;\setminus \;\Delta$
where
$$
\Delta := \{(x_1, \cdots , x_h) \in \times_{i=1}^h C\; | \; x_i = x_j \; \mbox{for some} \; i \neq j \}.
$$
To be more precise, we think of $\times_{i=1}^h C$ as the set of $h$-tuples of stable bundles of rank $\frac{n}{h}$ and 
degree $\frac{d}{h}$. If $\cU$ denotes a Poincar\'e bundle on $C \times C$ of rank $\frac{n}{h}$ and degree $\frac{d}{h}$ and 
$p_2$ the second projection of $C \times C$, denote
$$
\cF := \times_{i=1}^h P(p_{2*} \cU)|_{\times_{i=1}^h C \,\setminus \,\Delta}.
$$
With these notations the following proposition is obvious.

\begin{prop} \label{prop3.1}
There is a natural action of the symmetric group $S(h)$ on $\cF$ such that the quotient is isomorphic 
to the open subset $G_0^1$ of $G_0(n,d,1)$, 
$$
G_0^1 \simeq \cF/S(h).
$$
\begin{flushright}  $\square$ \end{flushright}
\end{prop} 

From this we conclude

\begin{prop} \label{prop3.2}
The birational type of $G_0(n,d,1)$ is independent of $n$ provided $d$ and $\emph{\Gcd}(n,d)$ are fixed.
\end{prop}

\begin{proof}
The sheaf $p_{2*}\cU$ is locally trivial over $C$. This implies that there is an open dense subset 
$U$ of $C$ such that
$$
\cF|_{\times_{i=1}^h U \;\setminus \,\Delta_U} \simeq \times_{i=1}^h(U \times \PP^{\frac{d}{h}-1})|_
{\times_{i=1}^h U  \; \setminus \, \Delta_U},
$$
where $\Delta_U:=\Delta\cap\times^h_{i=1}U$. Moreover
the action of the group $S(h)$ on this is given by permuting the factors. This implies the assertion.
\end{proof}

The following corollary improves \cite[Theorem 5.2]{ht} slightly in the case $k=1$.

\begin{cor}
Suppose $d = \prod_{i=1}^r p_i^{a_i}$ is the prime decomposition of $d$. Then there are at most $\prod_{i=1}^r (a_i + 1)$
birational types of varieties $G(\alpha;n,d,1)$. 
\end{cor}

\begin{proof}
According to \cite[Theorem 4.4 (ii)]{ln} the birational type of the variety $G(\alpha;n,d,1)$ does not depend on $\alpha$.
Hence Proposition \ref{prop3.2} implies that the number of birational types of $G(\alpha;n,d,1)$ equals 
at most the number of divisors of $d$.
\end{proof}

Fix $N$ in $\Pic^d(C)$ and define
$$
G_0^1(N) = G_0^1 \cap G_0(n,N,1).
$$
If $(E,V) \in G_0^1(N)$, then $E = F_1 \oplus \cdots \oplus F_h$ as in \eqref{eqn1} and $\otimes_{i=1}^h \det F_i = N$.

Regarding $\times_{i=1}^h C$ as the set of $h$-tuples of stable bundles $(F_1, \ldots , F_h)$ of rank $\frac{n}{h}$ 
and degree $\frac{d}{h}$, the subset
$$
C_N := \{(F_1, \ldots, F_h) \in \times_{i=1}^h C \;|\; \otimes_{i=1}^h \det F_i = N \}
$$ 
is isomorphic to $\times_{i=1}^{h-1} C$. The variety
$$
\cF_N := \cF|_{C_N \setminus (\Delta \cap C_N)}
$$
is a $(\times_{i=1}^h \PP^{\frac{d}{2}-1})$-fibration over $C_N \setminus(\Delta \cap C_N)$.
With this notation we have

\begin{prop} \label{prop3.4}
There is a natural action of the group $S(h)$ on $\cF_N$ such that the quotient is isomorphic to the open subset $G_0^1(N)$
of the variety $G_0(n,N,1)$,
$$
G_0^1(N) \simeq \cF_N/S(h).
$$
\end{prop}

\begin{proof}
The natural actions of $S(h)$ on the varieties $\times_{i=1}^h C$ and $\cF$ restrict to actions on $C_N$ and $\cF_N$ respectively.
Moreover $\Delta \cap C_N = \{ (F_1,\ldots,F_h) \in C_N \;|\; F_i \simeq F_j \;\mbox{for some}\; i \neq j \}$. The result now follows.
\end{proof}

\section{$G_0(n,d,1)$ with $\Gcd(n,d)=2$}

\subsection{The set up}
Now suppose $\Gcd(n,d) =2$ and consider the moduli space $G_0(n,d,1)$ as well as
its subspace $G_0(n,N,1)$.   
Then there are no stable vector bundles of rank $n$ and degree $d$. 
Hence, if $(E,V) \in G_0(n,d,1)$, then $E$ is of one of the following types 
\begin{enumerate}
\item $E = F_1 \oplus F_2$ with $F_1,F_2$ stable of rank $\frac{n}{2}$, degree $\frac{d}{2}$ and $F_1 \not \simeq F_2$;
\item there is a nontrivial exact sequence $0 \ra F \ra E \ra F \ra 0$ with stable $F$;
\item $E = F \oplus F$ with stable $F$. 
\end{enumerate}
The coherent systems of type (1) form the open set $G_0^1$ (respectively $G_0^1(N)$) of the previous section 
in the moduli space $G_0(n,d,1)$ (respectively $G_0(n,N,1)$). 
The coherent systems of type (2) form a locally closed subset 
$G_0^2$ (respectively $G_0^2(N)$) in $G_0(n,d,1)$ (respectively $G_0(n,N,1)$), whose boundary is the closed subset
$G_0^3$ (respectively $G_0^3(N)$) of coherent systems of type (3) in $G_0(n,d,1)$ (respectively $G_0(n,N,1)$).
Moreover we have stratifications
\begin{equation} \label{eq1}
G_0(n,d,1) = \sqcup_{i=1}^3 G_0^i \quad \mbox{and} \quad G_0(n,N,1) = \sqcup_{i=1}^3 G_0^i(N).
\end{equation}

\subsection{The spaces $G_0^1$ and $G_0^1(N)$}

According to Proposition \ref{prop3.1} there is a natural action of the group $\ZZ_2$ on the variety $\cF$ such that 
$$
G_0^1 \simeq \cF/\ZZ_2,
$$
where $\cF = P \times P \setminus (p \times p)^{-1}(\Delta)$; here $p: P \ra C$ is the projection of a 
$\PP^{\frac{d}{2}-1}$-bundle over the curve $C$ and $\Delta \subset C \times C$ the diagonal. This allows 
us to compute the Hodge polynomial of $G_0^1$.

\begin{prop} \label{prop4.1} \quad \newline
$\e_{G_0^1}(u,v) = \frac{(1+u)(1+v)(1-(uv)^{\frac{d}{2}})}{(1-uv)^2(1+uv)}\{(u+v)(uv-(uv)^{\frac{d}{2}}) +uv(1 -(uv)^{\frac{d}{2}+1}) \}.$
\end{prop}

\begin{proof} 
The above description of $\cF$ gives the following exact sequence of the $(+1)$-eigenspaces of the action of $\ZZ_2$ on $\cF$
on the $i$-th cohomology with compact support
$$
0 \ra H^i_c(\cF)_+ \ra H^i(P \times P)_+ \ra H^i((p \times p)^{-1}(\Delta))_+ \ra 0.
$$
This implies 
$$
\e_{G_0^1}(u,v) = \e_{\cF}(u,v)_+ = \e_{P \times P}(u,v)_+ - \e_{(p \times p)^{-1}(\Delta)}(u,v)_+.
$$
Since $\e_{P}(u,v) = \frac{(1+u)(1+v)(1-(uv)^{\frac{d}{2}})}{1-uv}$, \cite[Lemma 2.6]{mov} implies\\
$$
\begin{array}{rll}
\e_{P \times P}(u,v)_+ &= & \frac{1}{2} [\e_{P}(u,v)^2 + \e_{P}(-u^2,-v^2)]\\
& = & \frac{1}{2}\{\frac{(1+u)^2(1+v)^2(1-(uv)^{\frac{d}{2}})^2}{(1-uv)^2} 
+ \frac{(1-u^2)(1-v^2)(1-(uv)^d)}{1-(uv)^2}\}.
\end{array}
$$
Now $(p \times p)^{-1}(\Delta)$ is a $(\PP^{\frac{d}{2}-1} \times \PP^{\frac{d}{2}-1})$-fibration over $\Delta$ and $\ZZ_2$ acts
on it by swapping the two $\PP^{\frac{d}{2}-1}$'s. So, again using \cite[Lemma 2.6]{mov},
$$
\begin{array}{rcl}
\e_{(p \times p)^{-1} (\Delta)}(u,v)_+ & = &\e_{\Delta}(u,v) \e_{\PP^{\frac{d}{2}-1} \times \PP^{\frac{d}{2}-1}}(u,v)_+\\
& = &\frac{1}{2}(1+u)(1+v)\{ \frac{(1-(uv)^{\frac{d}{2}})^2}{(1-uv)^2} + \frac{1-(uv)^d}{1-(uv)^2} \}.
\end{array}
$$
Subtracting and simplifying we get the assertion.
\end{proof}

Now fix $N$ in $\Pic^d(C)$. Then $C_N \subset C \times C$ and $\Delta\cap C_N$ consists of 4 points (given by the 
square roots of $N$). Moreover $ \Delta\cap C_N$ is the fixed point set for the action of $S(2) = \ZZ_2$ on $C_N$ and 
$C_N/\ZZ_2$ is isomorphic to $\PP^1$. Here Proposition \ref{prop3.4} gives

\begin{prop} \label{prop4.2}
There is a natural action of the group $\ZZ_2$ on $\cF_N$ such that the quotient is isomorphic to the open subset $G_0^1(N)$
of the variety $G_0(n,N,1)$,
$$
G_0^1(N) \simeq \cF_N/\ZZ_2.
$$
\begin{flushright} $\square$ \end{flushright}
\end{prop}

Proposition \ref{prop4.2} allows us to compute the Hodge polynomial of $G_0^1(N)$.

\begin{prop}  \label{prop4.3}
$$
\e_{G_0^1(N)}(u,v) = \frac{1-(uv)^{\frac{d}{2}}}{(1-uv)^2(1+uv)}[(1-(uv)^{\frac{d}{2} + 1})(uv - 3) + (uv - (uv)^{\frac{d}{2}})(u+v)].
$$
\end{prop}

\begin{proof} From Proposition \ref{prop4.2} we conclude
\begin{eqnarray*}
\e_{G_0^1(N)}(u,v) & = & \e_{\cF_N}(u,v)_+\\
&=& \e_{C_N \setminus(\Delta \cap C_N)}(u,v)_+  \cdot \e_{\PP^{\frac{d}{2}-1} \times \PP^{\frac{d}{2}-1}}(u,v)_+ \\
&& \hspace*{1cm} + \e_{C_N \setminus(\Delta \cap C_N)}(u,v)_- \cdot \e_{\PP^{\frac{d}{2}-1} \times \PP^{\frac{d}{2}-1}}(u,v)_-
\end{eqnarray*}
Now $\e_{C_N \setminus(\Delta \cap C_N)}(u,v) = (1+u)(1+v) - 4$ and 
$$
\e_{C_N \setminus(\Delta \cap C_N)}(u,v)_+ = \e_{\PP^1 \setminus \{4 \; {\rm points}\}}(u,v) = uv - 3;
$$ 
hence
$$
\e_{C_N \setminus(\Delta \cap C_N)}(u,v)_- = u + v.
$$
Using \cite[Lemma 2.6]{mov} we get
\begin{eqnarray*}
\e_{\PP^{\frac{d}{2}-1} \times \PP^{\frac{d}{2}-1}}(u,v)_+ &=& \frac{1}{2}\{ \e_{\PP^{\frac{d}{2}-1}}(u,v)^2 + \e_{\PP^{\frac{d}{2}-1}}(-u^2,-v^2) \}\\
&=& \frac{1}{2} \left\{ \frac{(1 - (uv)^{\frac{d}{2}})^2}{(1-uv)^2} + \frac{1-(uv)^d}{1-(uv)^2} \right\}\\
&=& \frac{(1-(uv)^{\frac{d}{2}})(1 - (uv)^{\frac{d}{2} +1})}{(1-uv)^2(1+uv)}
\end{eqnarray*}
and similarly
$$
\e_{\PP^{\frac{d}{2}-1}\times \PP^{\frac{d}{2}-1}}(u,v)_- = \frac{(1-(uv)^{\frac{d}{2}})(uv - (uv)^{\frac{d}{2}})}{(1-uv)^2(1+uv)}.
$$
Inserting these expressions into the above formula for $\e_{G_0^1(N)}(u,v)$ and simplifying we get the assertion.
\end{proof}

\subsection{The spaces $G_0^2$ and $G_0^2(N)$} Recall that $(E,V) \in G_0(n,d,1)$ is of type (2) if it 
admits a non-trivial exact sequence $0 \ra F \ra E \ra F \ra 0$ with a stable vector bundle $F$. The bundle $F$ is
uniquely determined by $E$ and called the stable bundle associated to $E$. 
Conversely $E$ is uniquely determined by $F$, since $\dim \Ext^1(F,F) = h^1(End(F)) = h^0(End(F)) = 1$. 

\begin{lem}\label{new}
Fix a stable bundle $F$ of rank $\frac{n}{2}$ and degree $\frac{d}{2}$.
The variety of coherent systems $(E,V)$ of type (2) with associated stable bundle $F$ 
is isomorphic to an $\mathbb A^{\frac{d}{2}-1}$-fibration over the projective space $\PP^{\frac{d}{2}-1}$. 
\end{lem}

\begin{proof}
A coherent system $(E,V)$ with $\det E = (\det F)^2$ is of type (2) if and only if 
there is a non-trivial exact sequence $0 \ra F \ra E \stackrel{p}{\ra} F \ra 0$ such that $V \subset H^0(E)$ is a line which 
projects to a line $\ell$ in $H^0(F)$. (Note that, if $H^0(p)(V)=0$, the subsystem $(F,V)$ would contradict the $0^+$-stability of $(E,V)$.) Via the exact sequence the line $\ell$ is contained in 
$H^0(E): \; \ell \subset H^0(F) \subset H^0(E)$. So we end up with a plane in $H^0(E)$ containing $\ell$ and projecting to 
$\ell$. 

Considering lines as points in the corresponding projective spaces, we have 
$\ell \in P(H^0(F))$ and $V \in  P(H^0(p)^{-1}(\ell)) \subset P(H^0(E))$.
Hence the coherent systems $(E,V)$ of type (2) associated to the bundle $F$ with a fixed 
$\ell \in P(H^0(F))$ correspond to the points of $P(H^0(p)^{-1}(\ell)) \setminus P(H^0(F))$. 
Since $\mbox{Aut} E = \{ \lambda \cdot id_E \;|\; \lambda \in \CC^* \} \times  \End F$,
two coherent systems of this form are isomorphic if and only if the corresponding points of 
$P(H^0(p)^{-1}(\ell)) \setminus P(H^0(F))$ lie in the same line through $\ell$. 

Now choose a hyperplane $H$ in $P(H^0(p)^{-1}(\ell))$
not containing the point $\ell$. Then we can identify the isomorphism classes of pairs $(E,V)$ with fixed $F$, $E$ 
and $H^0(p)(V) = \ell$ with the points of $H \setminus (H \cap P(H^0(F)))$. Since $P(H^0(p)^{-1}(\ell))$ is of dimension $\frac{d}{2}$, 
this is an affine space $\mathbb A^{\frac{d}{2}-1}$ of dimension $\frac{d}{2} -1$. Since $\ell$ varies in 
the projective space $P(H^0(F))$, this completes the proof of the lemma.
\end{proof}

\begin{prop} \label{prop4.5}
\emph{(a)} The variety $G_0^2$ is isomorphic to an $\mathbb A^{\frac{d}{2}-1}$-fibration over a $\PP^{\frac{d}{2}-1}$-bundle 
over the curve $C$.\\
\emph{(b)} \quad $$\e_{G_0^2}(u,v) = (1+u)(1+v)(uv)^{\frac{d}{2}-1}\frac{1-(uv)^{\frac{d}{2}}}{1-uv}.$$
\end{prop}

\begin{proof} (a): Identifying $C$ with the moduli space of stable bundles of rank $\frac{n}{2}$ and degree $\frac{d}{2}$,
the natural map $G_0^2 \ra C$ is a morphism whose fibres are isomorphic to $\mathbb A^{\frac{d}{2}-1}$-fibrations 
over $\PP^{\frac{d}{2}-1}$ by Lemma \ref{new}.\\
(b): Since any $\mathbb A^{\frac{d}{2} - 1}$-fibration over a projective variety is locally trivial in the complex topology, 
we conclude from the third property of Hodge polynomials in section 2 that
$$
\e_{G_0^2}(u,v) = \e_C(u,v) \cdot \e_{\mathbb A^{\frac{d}{2}-1}}(u,v) \cdot \e_{\PP^{\frac{d}{2} - 1}}(u,v), 
$$
which gives the assertion, since $\e_{\mathbb A^{\frac{d}{2}-1}}(u,v) = (uv)^{\frac{d}{2}-1}$.
\end{proof}

\begin{prop} \label{prop4.6}
\emph{(a)} The variety $G_0^2(N)$ is isomorphic to the disjoint union of four  
$\mathbb A^{\frac{d}{2} - 1}$-fibrations over $\PP^{\frac{d}{2}-1}$.\\ 
\emph{(b)} 
$$
\e_{G_0^2(N)}(u,v) = 4 \frac{(uv)^{\frac{d}{2}-1}(1-(uv)^{\frac{d}{2}})}{1-uv}.
$$
\end{prop}

\begin{proof}
(a) follows from the fact that there are exactly 4 vector bundles $E$ providing coherent systems $(E,V)$ in $G_0^2(N)$,
since $\Pic^0(C)$ has exactly 4 two-division points. The proof of (b) is analogous to the proof of Proposition \ref{prop4.5} (b).
\end{proof}

\subsection{The spaces $G_0^3$ and $G_0^3(N)$}

\begin{prop} \label{prop4.7}
\emph{(a)} The variety $G_0^3$ is isomorphic to a ${\rm Gr}(2,\frac{d}{2})$-fibration over the curve $C$.\\
\emph{(b)} $$
\e_{G_0^3}(u,v) = (1+u)(1+v)\frac{1-(uv)^{\frac{d}{2}}}{(1-uv)^2(1+uv)}(1-(uv)^{\frac{d}{2}-1}).
$$
\end{prop}

\begin{proof}
Let $(E,V)\in G_0^3$. Then $E = F \oplus F$ for some stable bundle $F$ 
and $V$ is 
generated by a section $\sigma = (\sigma_1,\sigma_2)$ with $\sigma_1, \sigma_2 \in H^0(F)$. Moreover $\sigma_1$, $\sigma_2$ are linearly independent (otherwise there would exist a subsystem $(F,V)$ of $(E,V)$ with $\rk F=1$, contradicting the $0^+$-stability of $(E,V)$).
Hence giving a line $V$ in $H^0(E)$ is equivalent to giving a two-dimensional subspace of $H^0(F)$. 
Identifying the space of stable bundles of rank $\frac{n}{2}$ and degree $\frac{d}{2}$ with the curve $C$ gives assertion (a).
For the proof of (b) we use the fact that any Gr$(2,\frac{d}{2})$-fibration over the curve $C$ is locally trivial in the Zariski topology.
\end{proof}

\begin{prop} \label{prop4.8}
\emph{(a)} The variety $G_0^3(N)$ is isomorphic to the disjoint union of four Grassmannians ${\rm Gr}(2,H^0(F))$:
$$
G_0^3(N) \simeq \sqcup_{(\det F)^2 = N} {\rm Gr}(2,H^0(F)).
$$
\emph{(b)}
$$
\e_{G_0^3(N)}(u,v) = 4 \frac{1-(uv)^{\frac{d}{2}}}{(1-uv)^2(1+uv)}(1-(uv)^{\frac{d}{2}-1}).
$$
\end{prop}

\begin{proof}
The proof is the same as the proof of the previous proposition using the fact that the line bundle $N$ admits exactly 4 square roots.
\end{proof}

\subsection{The Hodge polynomials of $G_0(n,d,1)$ and  $G_0(n,N,1)$} The Hodge polynomial is additive on disjoint unions. 
Hence we conclude from equation \eqref{eq1} adding the formulas of 
Propositions \ref{prop4.1}, \ref{prop4.5} and \ref{prop4.7} (respectively Propositions \ref{prop4.3}, \ref{prop4.6} and \ref{prop4.8})
after multiplying out and simplifying,

\begin{theorem} \label{thm4.9}
\emph{(a)}\\
$\e_{G_0(n,d,1)}(u,v) =$
$$ 
= \frac{(1+u)(1+v)(1-(uv)^{\frac{d}{2}})}{(1-uv)^2(1+uv)}[(u+v)(uv -(uv)^{\frac{d}{2}}) + (1+uv)(1-(uv)^{\frac{d}{2}+1})];
$$
\emph{(b)}\\
$\e_{G_0(n,N,1)}(u,v) =$
$$ 
= \frac{1-(uv)^{\frac{d}{2}}}{(1-uv)^2(1+uv)}[(u+v)(uv -(uv)^{\frac{d}{2}}) + (1+uv)(1-(uv)^{\frac{d}{2}+1})].
$$
\begin{flushright} $\square$ \end{flushright}
\end{theorem}

\begin{rem} \emph{From Theorem \ref{thm4.9} we see that} 
$$
\e_{G_0(n,d,1)} = \e_{G_0(n,N,1)} \cdot \e_C.
$$
\emph{However the morphism}
$$
G_0(n,d,1) \ra C, \quad (E,V) \mapsto \det E,
$$
\emph{whose fibre over} $N \in C$ \em{is} $G_0(n,N,1)$, is not Zariski locally trivial. This follows from the fact that 
$\e_{G_0^1} \neq \e_{G_0^1(N)} \cdot \e_C$ (see Propositions \ref{prop4.1} and \ref{prop4.3}).
\end{rem}

\section{The Hodge polynomial of $G(\alpha;2+ad,d,1)$} 

As outlined in \cite[Section 6]{ln} and described in more generality in \cite{bgn}, 
the moduli spaces $G(\alpha;2+ad,d,1)$ are modified at certain critical values 
$\alpha = \alpha_i$ ($1\le i\le L$) which in this case are given by the following lemma.

\begin{lem}  \label{lemma5.1}
The critical values for coherent systems of type $(2+ad,d,1)$ are given by 
$$
\alpha_i = \frac{d-2d_1}{1+a(d-d_1)}
$$
where $d_1 = [\frac{d-1}{2}] -i+1$ and $0 < i \leq  [\frac{d-1}{2}]$.
\end{lem}

\begin{proof}
According to \cite[section 6]{ln} we have
$$
\alpha_i = \frac{n_1d_2-n_2d_1}{n_2}
$$
where $n_1,n_2,d_1,d_2$ are positive integers such that $n_1+n_2 = 2 +ad$ and $d_1+d_2= d$. Moreover $\frac{d_1}{n_1} < \frac{d_2}{n_2}$ and 
$$
0< \alpha_i < \frac{d}{1+ad}.
$$
After substituting $n_2 = 2 + ad -n_1$ and $d_2 = d-d_1$, we get
$$
\alpha_i = \frac{n_1d-d_1(2+ad)}{2+ad-n_1}.
$$
So $\alpha_i > 0$ is equivalent to
$$
n_1 > \frac{2d_1}{d} + d_1a
$$
and $\alpha_i < \frac{d}{1+ad}$ is equivalent to $(n_1d-d_1(2+ad))(1+ad)<d(2+ad-n_1)$ which simplifies to
$$
n_1 < \frac{d_1}{d} + d_1a + 1.
$$
So the only possible value for $n_1$ is 
\begin{equation} \label{eqn3}
 n_1 = d_1a + 1
\end{equation}
and this satisfies the inequalities if and only if $d_1 < \frac{d}{2}$. Writing
\begin{equation} \label{eqn4}
i = \left[ \frac{d-1}{2} \right] - d_1 + 1,
\end{equation}
we obtain the $\alpha_i$ in increasing order of magnitude.
\end{proof}

The moduli spaces $G(\alpha;2+ad,d,1)$ for $\alpha_i < \alpha < \alpha_{i+1}$ are denoted by $G_i$. Note that 
$$
\alpha_L = \frac{d-2}{1+a(d-1)}
$$
and $G(\alpha;2+ad,d,1) = \emptyset$ for $\alpha \geq \frac{d}{1+ad}$. So $G_L$ denotes $G(\alpha;2+ad,d,1)$ for 
$\alpha_L < \alpha < \frac{d}{1+ad}$.

According to \cite[Remark 5.5]{bgn} the moduli space $G_L$ is isomorphic to a $\PP^{d-1}$-bundle over the curve $C$.
Hence
\begin{equation}  \label{eq2}
\e_{G_L}(u,v) = (1+u)(1+v)\frac{1-(uv)^d}{1-uv}
\end{equation}
and for any line bundle $N$ of degree $d$,
\begin{equation}  \label{eq3}
\e_{G_L}(u,v) = \frac{1-(uv)^d}{1-uv}.
\end{equation}

The modifications at $\alpha_i$ are given as follows. There are closed subvarieties $G_i^+$ of $G_i$ and $G_i^-$ of $G_{i-1}$ such that
\begin{equation} \label{eq4}
G_i \setminus G_i^+ \simeq G_{i-1} \setminus G_i^-.
\end{equation} 
These subvarieties are called {\it flip loci} and are described as follows.
The variety $G_i^+$ is given by coherent systems 
$(E,V)$ which are non-trivial extensions
\begin{equation} \label{eq5}
0 \ra (F_2,0) \ra (E,V) \ra (F_1,V_1) \ra 0
\end{equation}
where $F_2$ is a stable bundle of rank $n_2$ and degree $d_2$ and $(F_1,V_1)$ belongs to the moduli space $G(\alpha_i^+;n_1,d_1,1)$ of coherent systems of type $(n_1,d_1,1)$ which are $\alpha$-stable for $\alpha$ slightly greater than $\alpha_i$. Note that by (\ref{eqn3}),
$$ 
n_1 = d_1a +1, \quad  \quad n_2  = d_2a + 1.
$$
Hence $\Gcd(n_2,d_2) = 1$. So the moduli space of the bundles 
$F_2$ can be identified with the curve $C$.

\begin{lem}  \label{lem5.2}
There are no critical values for coherent systems of type $(1+ad_1,d_1,1)$.
\end{lem}

\begin{proof}
According to \cite[section 6]{ln} any critical value is of the form
$$
\alpha = \frac{m(d_1-e)-(ad_1+1-m)e}{ad_1+1-m}
$$
with $0 < m < ad_1+1$ and $ 0 < e < d_1$. Moreover
$$
0 < \alpha < \frac{d_1}{1+ad_1-1} = \frac{1}{a}.
$$
The first inequality is equivalent to 
$$
m > ae + \frac{e}{d_1}
$$
and the second inequality is equivalent to 
$$
m < ae +1,
$$
which gives a contradiction.
\end{proof}

\begin{rem}
\emph{It follows from Proposition \ref{prop2.1} and Lemma \ref{lem5.2} that $G_L(1+ad_1,d_1,1)$ is a $\PP^{d_1-1}$-bundle over $C$.
This is by no means clear from the general structure theorem for $G_L$ (\cite[Theorem 5.4]{bgn})}.
\end{rem}

\begin{prop} \label{prop5.3}
The flip locus $G_i^+$ is a 
$\PP^{d_1 - 1}$-bundle over $C \times G_0(1+ad_1,d_1,1)$ and $G_0(1+ad_1,d_1,1)$ is a $\PP^{d_1-1}$-bundle over $C$.
\end{prop}

\begin{proof}
By Lemma \ref{lem5.2} the variety $G(\alpha_i^+;1+ad_1,d_1,1)$ is isomorphic to $G_0(1+ad_1,d_1,1)$ 
which according to Proposition \ref{prop2.1} is a $\PP^{d_1 - 1}$-bundle over $C$. The result now follows from \eqref{eq5} provided 
that
$$
\dim \Ext^1((F_1,V_1),(F_2,0)) = d_1.
$$ 
Now this dimension is given by
$$
C_{12} + \mathbb H^0_{12} + \mathbb H^2_{12}
$$
(see \cite[Proposition 3.2]{bgn}), where $\mathbb H^0_{12} = \Hom((F_1,V_1),(F_2,0))$ and 
$\mathbb H^2_{12} = H^0(F_2^* \otimes N_1)^*$ with $N_1 = \mbox{Ker}(V_1 \otimes \cO_C \ra F_1)$.
Since $(F_1,V_1)$ and $(F_2,0)$ are both $\alpha_i$-stable and non-isomorphic of the same $\alpha_i$-slope, we get
$\mathbb H_{12}^0 = 0$. It is obvious that $N_1=0$ and hence $\mathbb H_{12}^2 = 0$. This completes the proof, since $C_{12} = d_1$ by \cite[equation (13)]{ln}.
\end{proof}

The variety $G_i^-$ is given by coherent systems
$(E,V)$ which are nontrivial extensions
\begin{equation}   \label{eq6}
0 \ra (F_1,V_1) \ra (E,V) \ra (F_2,0) \ra 0
\end{equation}
where $(F_1,V_1)$ and $F_2$ are as above. 

\begin{prop} \label{prop5.4}
The flip locus $G_i^-$ is a 
$\PP^{d - 2d_1 - 1}$-bundle over $C \times G_0(1+ad_1,d_1,1)$ and $G_0(1+ad_1,d_1,1)$ is a $\PP^{d_1-1}$-bundle over $C$.
\end{prop}

\begin{proof}
The proof is the same as the proof of Proposition \ref{prop5.3} with $C_{12}$ replaced by $C_{21}$, 
which is equal to $-d_1+d_2 = d - 2d_1$ according to \cite[equation (13)]{ln}.
\end{proof}

\begin{lem}  \label{lem5.5}
For every $i = 1, \ldots, L=\left[ \frac{d-1}{2} \right]$ we have
$$
\e_{G_{i-1}}(u,v) = \e_{G_i}(u,v) + \frac{(1+u)^2(1+v)^2(1-(uv)^{d_1})}{(1-uv)^2} \{ (uv)^{d_1} - (uv)^{d-2d_1} \}.
$$
\end{lem}

\begin{proof}
According to \eqref{eq4} and Propositions \ref{prop5.3} and \ref{prop5.4} we have
$$
\begin{array}{rcl}
\e_{G_{i-1}} -  \e_{G_i} & = & \e_{G_i^-} - \e_{G_i^+} \\
&=& (1+u)^2(1+v)^2 \frac{1-(uv)^{d_1}}{1-uv}\{ \frac{1-(uv)^{d-2d_1}}{1-uv} - \frac{1-(uv)^{d_1}}{1-uv} \}.
\end{array}
$$
This gives the assertion.
\end{proof}

\begin{theorem} \label{thm5.6}
For $i = 0, \ldots, L$ we have\\ 
$ \displaystyle{\e_{G_i}(u,v) = (1+u)(1+v)\frac{1-(uv)^d}{1-uv} \;+}$
$$
+ \frac{(1+u)^2(1+v)^2(1-(uv)^{\frac{d-\gamma}{2} -i)}}{(1-uv)^2(1-(uv)^2)} 
(uv - (uv)^{\gamma + 2i})(1 - (uv)^{\frac{d-\gamma}{2}-i+1}). 
$$
where $\gamma$ is $1$ if $d$ is odd and $2$ if $d$ is even.
\end{theorem}

\begin{proof}
By Lemmas \ref{lemma5.1} and \ref{lem5.5} and downwards induction on $i$ we have
$$
\begin{array}{rcl}
\e_{G_i} 
&=& \e_{G_L} + \frac{(1+u)^2(1+v)^2}{(1-uv)^2} 
\sum_{d_1=1}^{\frac{d- \gamma}{2} -i}(1-(uv)^{d_1})((uv)^{d_1} - (uv)^{d-2d_1})).
\end{array}
$$
Now the sum equals\\
$ \sum_{d_1=1}^{\frac{d- \gamma}{2} -i}[(uv)^{d_1} - (uv)^{2d_1} - (uv)^{d -2d_1} + (uv)^{d-d_1}]
= \frac{uv(1-(uv)^{\frac{d-\gamma}{2}-i})}{1-uv} - \\
\hspace*{0.5cm} - \frac{(uv)^2(1-(uv)^{d-\gamma -2i})}{1-(uv)^2}
- \frac{(uv)^{\gamma +2i}(1-(uv)^{d - \gamma -2i})}{1-(uv)^2} 
+ \frac{(uv)^{\frac{d+\gamma}{2} +i}(1-(uv)^{\frac{d-\gamma}{2}-i})}{1-uv}\\
\hspace*{1.5cm} = \frac{1-(uv)^{\frac{d-\gamma}{2}-i}}{1-uv}(uv + (uv)^{\frac{d+\gamma}{2} +i}) - 
\frac{1-(uv)^{d-\gamma -2i}}{1-(uv)^2}((uv)^2 +(uv)^{\gamma + 2i})\\
\hspace*{1.5cm} = \frac{1-(uv)^{\frac{d-\gamma}{2} -i}}{1-(uv)^2} (uv - (uv)^{\gamma + 2i})(1 - (uv)^{\frac{d-\gamma}{2}-i+1}).\\
$
Together with \eqref{eq2} this gives the assertion.
\end{proof}

\begin{rem}
\emph{Note that the formula of the theorem is independent of $a$. This is consistent with the main theorem of \cite{ht}
which implies that the isomorphism class of $G_0(2+ad,d,1)$ is independent of $a$. Our result suggests 
that the isomorphism class of $G_i$ should be independent of $a$ and we prove this in the next section.}
\end{rem}

\begin{rem}
\emph{When $i=0$ and $d$ is odd, the second term in the formula of Theorem \ref{thm5.6} is $0$, so we have
$$
\epsilon_{G_0}(u,v)=(1+u)(1+v)\frac{1-(uv)^d}{1-uv}
$$
in accordance with Proposition \ref{prop2.1}(2). For $i=0$ and $d$ even, it can easily be checked that the formula of Theorem \ref{thm5.6} agrees with Theorem \ref{thm4.9}(a).}
\end{rem}

For coherent systems of fixed determinant $N$ we can define moduli spaces $G_i(N)$ in the same way as we defined the moduli
spaces $G_i$ at the beginning of this section. 

\begin{prop} \label{prop5.9}
For $i = 0, \ldots, L$ we have\\ 
$ \displaystyle{\e_{G_i(N)}(u,v) = \frac{1-(uv)^d}{1-uv} \;+}$
$$
+ \frac{(1+u)(1+v)(1-(uv)^{\frac{d-\gamma}{2} -i)}}{(1-uv)^2(1-(uv)^2)} 
(uv - (uv)^{\gamma + 2i})(1 - (uv)^{\frac{d-\gamma}{2}-i+1}). 
$$
\end{prop}

\begin{proof}
Note that in \eqref{eq5} and \eqref{eq6} the bundle $F_2$ is uniquely determined by $F_1$ and the fact that $\det F_2 \simeq N \otimes ( \det F_1)^{-1}$. The rest of the proof of Theorem \ref{thm5.6} goes through exactly as before.
\end{proof}

\section{Isomorphism class of $G_i$}

For any positive integer $a$, let $\Phi_a$ denote the Fourier-Mukai transform defined in \cite[section 3.2]{ht}.
We use the results of the previous section in order to prove the following theorem.

\begin{theorem} \label{thm6.1}
The Fourier-Mukai transform $\Phi_a$ induces an isomorphism of moduli spaces
$$
\Phi_a^0: G_i(2,d,1) \ra G_i(2+ad,d,1)
$$
for every $i$.
\end{theorem}

\begin{rem}  \label{rem6.2}
\emph{In the case $i=0$ the theorem is a special case of the main result of \cite{ht} and we use this result in the proof.}
\end{rem}

\begin{proof}
The proof is by induction on $i$, the case $i=0$ being covered by Remark \ref{rem6.2}.

Now suppose $1 \leq i \leq L$ and assume that $\Phi_a^0: G_{i-1}(2,d,1) \ra G_{i-1}(2 +ad,d,1)$ 
is an isomorphism. We shall prove that the Fourier-Mukai transform $\Phi_a$ induces isomorphisms on $G_i^+(2,d,1)$ and 
$G_i^-(2,d,1)$ and this will complete the proof of the theorem.

The elements of $G_i^+$ are given by the exact sequences \eqref{eq5}, where $(F_1,V_1)$ is contained in 
$G_0(1,d_1,1)$ and $F_2$ is a line bundle. According to \cite[Theorem 4.5]{ht} $\Phi_a$ induces an isomorphism 
of moduli spaces $\Phi_a^0: G_0(1,d_1,1) \ra G_0(1+ad_1,d_1,1)$ and by \cite[Proposition 2.10 (1)]{ht} $\Phi_1^0(F_2)$ 
is stable of rank $1 + ad_2$ and degree $d_2$.

Since $F_2$ is $\Phi_a-IT_0$ in the sense of \cite{ht}, we get an exact sequence 
\begin{equation}  \label{eq11}
0 \ra \Phi_a^0(F_2,0) \ra \Phi_a^0(E,V) \ra \Phi_a^0(F_1,V_1) \ra 0.
\end{equation} 
Moreover $\Ext^1((F_1,V_1),(F_2,0)) \simeq \Ext^1(\Phi_a^0(F_1,V_1),\Phi_a^0(F_2,0))$. So we obtain a map
\begin{equation}  \label{eq12}
\Phi_a^0: G_i^+(2,d,1) \ra G_i^+(2+ad,d,1).
\end{equation}
Denote by $\Psi_a^1$ the Fourier-Mukai transform inverse to $\Phi_a^0$ as defined in \cite{ht}. It remains to show that
\begin{equation} \label{eq13}
\Psi_a^1: G_i^+(2+ad,d,1) \ra G_i^+(2,d,1),
\end{equation}
since then by the results of \cite{ht} the map \eqref{eq13} is the inverse of \eqref{eq12}.

So let $(E,V) \in G_i^+(2+ad,d,1)$. It is given by an exact sequence \eqref{eq5} with $(F_1,V_1) \in G(\alpha_i^+;1+ad_1,d_1,1)$ and $F_2$ 
stable of rank $1+ad_2$ and degree $d_2$. By Lemma \ref{lem5.2}, $(F_1,V_1) \in G_0(1+ad_1,d_1,1)$. Since $F_1$ and $F_2$ 
are both semistable of positive degree, they are $\Psi_a-IT_1$ in the sense of \cite{ht}. 
Hence by \cite[Proposition 3.7]{ht},
$\Psi_a^1(F_1,V_1)$ is a coherent system of type $(1,d_1,1)$ and $\Psi_a^1(F_2)$ is a line bundle. 
Moreover we have an exact sequence 
$$
0 \ra \Psi_a^1(F_2,0) \ra \Psi_a^1(E,V) \ra \Psi_a^1(F_1,V_1) \ra 0.
$$  
It follows that $\Psi_a^1(E,V) \in G_i^+(2,d,1)$ which implies the assertion.
This completes the proof for the varieties $G_i^+$. The proof for the $G_i^-$ is the same as for the $G_i^+$.
\end{proof}

\section{Birational type of $G(\alpha;n,d,k)$}

For $\Gcd (n,d) = 1$ and any $k$ we determined in Proposition \ref{prop2.1} and Corollary \ref{cor2.2} the birational 
type of the variety $G_0(n,d,k)$. Since the birational type of $G(\alpha;n,d,k)$ (respectively $G(\alpha;n,N,k)$) 
is independent of $\alpha$ 
(see \cite[Theorem 4.4 (ii)]{ln}), this determines the birational type of all the moduli spaces $G(\alpha;n,d,k)$ 
(respectively $G(\alpha;n,N,k))$  in this case. To summarize we have,

\begin{prop}\label{prop8.1}
If $\emph{\Gcd}(n,d) = 1$ and $k \leq d$, then $G(\alpha;n,N,k)$ is rational and $G(\alpha;n,d,k)$ is birational to $\PP^{k(d-k)} \times C$.
\end{prop}

\begin{proof}
The only remaining thing to observe is that $G_0(n,d,k)$ and $G_0(n,N,k)$ are non-empty in this case \cite[Proposition 3.2]{ln}.
\end{proof}

Now suppose $\Gcd(n,d) = 2$ and $k = 1$.

\begin{prop} \label{prop6.1}
If $\emph{\Gcd}(n,d) = 2$, then $G(\alpha;n,N,1)$ is rational and $G(\alpha;n,d,1)$ is birational to $\PP^{d-1} \times C$.
\end{prop}

\begin{proof}
By Proposition \ref{prop3.2} we have that $G_0(n,d,1)$ is birational to $G_0(2,d,1)$. 
Since the birational type is independent of $\alpha$, $G_0(2,d,1)$ is birational to $G_L(2,d,1)$ which is a locally 
trivial $\PP^{d-1}$-bundle over $C$ according to \cite[Remark 5.5]{bgn}. 
\end{proof}

\begin{prop}\label{prop8.3}
If $\emph{\Gcd}(n-k,d) = 1$ and $k < min(d,n)$, then $G(\alpha;n,N,k)$ is rational and $G(\alpha;n,d,k)$ is birational to 
$\PP^{k(d-k)} \times C$.
\end{prop}

\begin{proof}
By \cite[Remark 5.5]{bgn} the moduli space $G_L(n,d,k)$ is a Gr$(k,d)$-fibration over $C$ which is Zariski locally trivial.
Also $G_L(n,N,k) \simeq \mbox{Gr}(k,d)$. Now the result follows from the fact that the birational type is 
independent of $\alpha$. 
\end{proof}

For general $(n,d,k)$ with $\Gcd(n,d) = h>1$ and $k < d$ we have morphisms
$$
G_0(n,d,k) \ra \widetilde{M}(n,d) \quad \mbox{and} \quad G_0(n,N,k) \ra \widetilde{M}(n,N),
$$
where $\widetilde{M}(n,d)$ is the moduli space of S-equivalence classes of semistable bundles of rank 
$n$ and degree $d$ on $C$ and $\widetilde{M}(n,N)$ the subvariety of $\widetilde{M}(n,d)$ with fixed determinant $N$.
From \cite{tu} we have that $\widetilde{M}(n,N) \simeq \PP^{h-1}$ and $\widetilde{M}(n,d) \simeq S^hC$. With these notations, we have

\begin{theorem}\label{th8.4}
For all $\alpha$ for which $G(\alpha;n,d,k)\ne\emptyset$, \\
\emph{(1)} $G(\alpha;n,N,k)$ is birational to a variety $Y_N$, where $Y_N$ is fibred over $\PP^{h-1}$ with general fibre unirational.\\
\emph{(2)} $G(\alpha;n,d,k)$ is birational to a variety $Y$, where $Y$ is fibred over $S^hC$ with general fibre unirational.
\end{theorem}

\begin{proof}
The general point of $\widetilde{M}(n,d)$ is represented by a bundle of the form $F_1 \oplus \cdots \oplus F_h$, where the 
$F_i$ are non-isomorphic stable bundles of rank $\frac{n}{h}$ and degree $\frac{d}{h}$.
To obtain $(E,V) \in G_0(n,d,1)$ we must choose a subspace $V$ of $H^0(E)$ of dimension $k$. 
Since $G_0(n,d,k) \neq \emptyset$ 
if $k < d$ (see \cite{ln}) and $\alpha$-stability is an open condition, the subspace $V$ must belong to a
non-empty Zariski open subset of Gr$(k,H^0(E))$. The condition $(E,V) \simeq (E,V')$ means that $V$ and $V'$ 
are in the same orbit for the action of Aut$E$. A similar statement applies to $G_0(n,N,k)$.
The result follows from the fact that the birational type is independent of $\alpha$.
\end{proof}

\begin{rem}
\emph{So far as we know, it is possible that all the $G(\alpha;n,N,k)$ are rational varieties.}
\end{rem}

\end{document}